\documentclass[12pt,a4paper]{amsart}
\usepackage[a4paper, left=28mm, top=30mm, bottom=32mm]{geometry}
\usepackage{amssymb}
\usepackage{amsthm}
\usepackage{amsmath}
\usepackage{amscd}
\usepackage[all]{xy}
\usepackage{enumitem}
\usepackage[dvipdfmx]{hyperref}
\newtheorem{thm}{Theorem}[section]
\newtheorem{lem}[thm]{Lemma}
\newtheorem{conj}[thm]{Conjecture}

\theoremstyle{definition}

\newtheorem{rmk}[thm]{Remark}
\newtheorem{cor}[thm]{Corollary}
\numberwithin{equation}{section}
\newcommand{\var}{\overline}
\newcommand{\bb}{\mathbb}

\DeclareMathOperator{\NS}{NS}

\DeclareMathOperator{\pr}{pr}

\DeclareMathOperator{\rank}{rank}

\DeclareMathOperator{\cl}{cl}
\DeclareMathOperator{\id}{id}
\DeclareMathOperator{\Aut}{Aut}
\DeclareMathOperator{\Sur}{Sur}
\DeclareMathOperator{\Alb}{Alb}

\DeclareMathOperator{\Spec}{Spec}
\allowdisplaybreaks[1]
\title[Dynamical degree and arithmetic degree on product varieties]
{Dynamical degree and arithmetic degree of endomorphisms on product varieties}
\author{Kaoru Sano}
\address{Department of Mathematics, Faculty of Science, Kyoto University, Kyoto 606-8502, Japan}
\email{ksano@math.kyoto-u.ac.jp}

\begin{document}
\maketitle
\begin{abstract}
For a dominant rational self-map on a smooth projective variety
defined over a number field,
Shu Kawaguchi and Joseph H. Silverman conjectured that
the (first) dynamical degree is equal to the arithmetic degree
at an algebraic point whose forward orbit is well-defined and Zariski dense.
We give some examples of self-maps on product varieties and
rational points on them for which the Kawaguchi-Silverman conjecture holds.
\end{abstract}
\section{Introduction}\label{intro} 
Let $X$ be a smooth projective variety defined over a number field $k$,
and $f\colon X\dashrightarrow X$ a dominant rational self-map defined over $k$.
Let $I_f \subset X$ be the indeterminacy locus of $f$.
Let $X_f (\overline{k})$ be the set of $\overline{k}$-rational points $P$
on $X$ such that $f^n(P) \notin I_f$ for every $n$.
For a $\var{k}$-rational point $P\in X_f(\var{k}),$ its {\it forward $f$-orbit} is defined by
$\mathcal{O}_f(P):=\{ f^n(P):n\geq 0\}.$

Let $H$ be an ample divisor on $X$ defined over $k$.
The (first) {\it dynamical degree} of $f$ is defined by
$$\delta_f :=\lim_{n\to \infty} \deg(((f^n)^\ast H)\cdot H^{\dim X -1})^{1/n}.$$

The {\it arithmetic degree} of $f$ at a $\overline{k}$-rational point $P\in X_f(\overline{k})$
is defined by
$$\alpha _f(P):=\lim_{n\to \infty} h_H^+(f^n(P))^{1/n}$$
if the limit on the right hand side exists.
Here, $h_H\colon X(\overline{k})\longrightarrow [0,\infty )$
is the (absolute logarithmic) Weil height function associated with $H$,
and we put $h_H^+:=\max\{ h_H,1\}$. 

Shu Kawaguchi and Joseph H. Silverman formulated the following conjecture.

\begin{conj}[{Kawaguchi-Silverman conjecture (see \cite[Conjecture 6]{rat})}] \label{KS} 
For every $\overline{k}$-rational point $P\in X_f(\overline{k})$,
the arithmetic degree $\alpha_f(P)$ is defined.
Moreover, if the forward $f$-orbit $\mathcal{O} _f(P)$ is Zariski dense in $X$,
the arithmetic degree $\alpha _f(P)$ is equal to the dynamical degree $\delta_f$,
i.e., we have $$\alpha _f(P)=\delta _f.$$
\end{conj}

The existence of the limit defining the arithmetic degree
when $f$ is a dominant {\it endomorphism} (i.e., $f$ is defined everywhere) is proved in \cite{ab1}.
But in general, the convergence is not known.
It seems difficult to prove Conjecture \ref{KS} in full generality.

The aim of this paper is to give examples of endomorphisms
on product varieties and rational points on them
for which Conjecture \ref{KS} is true.

We prove Conjecture \ref{KS} in the following situations.

\begin{thm}\label{main1}
For $i=1,2,\ldots ,n$, let $X_i$ be a smooth projective variety
defined over a number field $k$.
Assume that each $X_i$ satisfies at least one of the following conditions:
\begin{itemize}
\item the first Betti number of $X_i(\bb{C})$ is zero
		and the N\'eron-Severi group of $X_i\otimes_k \var{k}$ has rank one,
\item $X_i$ is an abelian variety,
\item $X_i$ is an Enriques surface, or
\item $X_i$ is a $K3$ surface.
\end{itemize}
Then Conjecture \ref{KS} is true for any endomorphism
$f\colon \prod_{i=1}^n X_i \longrightarrow \prod_{i=1}^n X_i$
defined over $k$.
\end{thm}

We also prove that, when one of the direct factors is of general type,
any endomorphism on the product variety does not admit Zariski dense forward orbit.
Thus, Conjecture \ref{KS} is obviously true for such endomorphisms.

\begin{thm}\label{main2}
Let $X$ and $Y$ be smooth projective varieties of dimension $\geq 1$
defined over a subfield $k\subset \bb{C}$,
and $f \colon X \times Y \longrightarrow X\times Y$ an endomorphism defined over $k$.
Assume that at least one of $X$ or $Y$ is of general type.
Then, for every $\var{k}$-rational point $P \in (X \times Y)(\var{k})$,
the forward $f$-orbit $\mathcal{O}_f(P)$ is not Zariski dense in $X\times Y$.
\end{thm}

\begin{rmk}\label{results}
Kawaguchi and Silverman proved Conjecture \ref{KS} in the following cases
(for details, see \cite{ab1}, \cite{eg}, \cite{Gm}, \cite{ab2}).
\begin{itemize} 
\item (\cite[Theorem 2 (a)]{eg}) $f$ is an endomorphism and
		the N\'eron-Severi group of $X\otimes_k \var{k}$ has rank one.
\item (\cite[Theorem 2 (b)]{eg}) $f$ is the extension to $\mathbb{P} ^N$
		of a regular affine automorphism on $\mathbb{A} ^N$.
\item (\cite[Theorem 2 (c)]{eg}) $X$ is a smooth projective surface
		and $f$ is an automorphism on $X$.
\item (\cite[Proposition 19]{Gm}) $f$ is the extension to $\mathbb{P}^N$
		of a monomial endomorphism on $\mathbb{G}_m^N$
		and $P\in \mathbb{G} _m^N(\overline{k})$.
\item (\cite[Corollary 32]{ab1}, \cite[Theorem 2]{ab2}) $X$ is an abelian variety.
		Note that any rational map between abelian varieties is automatically a morphism,
		and can be written as the composition of
		a homomorphism of abelian varieties and the translation by a point.
\end{itemize}
\end{rmk}

\subsection*{Notation}
	The base field $k$ is always a number field
		or a subfield of $\bb{C}$.
		A variety defined over $k$ means a scheme
		of finite type over $\Spec k$ which is geometrically integral.
	An endomorphism on a variety $X$ means a morphism from $X$ to itself.
	For a smooth projective variety defined over a number field $k$,
		$b_1(X):=\dim_\bb{Q} H^1(X(\bb{C}),\bb{Q})$ denotes the first Betti number
		of the complex manifold $X(\bb{C})$, and
	$\NS(X)$ denotes the N\'eron-Severi group of $X\otimes_k \var{k}.$
	It is well-known that $b_1(X)$ does not depend on
	the choice of an embedding $k\hookrightarrow \bb{C}$, and
	$\NS(X)$ is a finitely generated abelian group.
	We put $\NS(X)_\bb{R} := \NS(X)\otimes_\bb{Z} \bb{R}.$
	The Albanese variety of $X$ is denoted by $\Alb(X)$.

\subsection*{Outline of this paper}
In Section \ref{recall}, we recall the definitions and some properties
of dynamical and arithmetic degrees.
In Section \ref{reduction}, we prove some reduction lemmata for Conjecture \ref{KS}.
In Section \ref{Splitting of endomorphisms},
we study some sufficient conditions for endomorphisms
on product varieties to be split.
These are important to prove Theorem \ref{main1} and Theorem \ref{main2}.
Theorem \ref{main1} is proved in Section \ref{cor1}, and
Theorem \ref{main2} is proved in Section \ref{cor2}.

\section{Dynamical degree and Arithmetic degree}\label{recall}
Let $H$ be an ample divisor on a smooth projective variety $X$
defined over a number field $k$.
The (first) {\it dynamical degree} of a dominant rational self-map $f\colon X \dashrightarrow X$
is defined by
$$\delta_f :=\lim_{n\to \infty} \deg(((f^n)^\ast H)\cdot H^{\dim X-1})^{1/n}.$$
The limit defining $\delta_f$ exists,
and $\delta _f$ does not depend on the choice of $H$ (see \cite[Corollary 7]{dinh}, \cite[Proposition 1.2]{guedj}).
Note that if $f$ is an endomorphism, we have $(f^n)^{\ast}=(f^\ast )^n$
as a linear self-map on $\NS (X)$.
But if $f$ is merely a rational self-map,
we have $(f^n)^{\ast}\neq (f^\ast )^n$ in general.

\begin{rmk}[{\cite[Proposition 1.2 (iii)]{guedj}, \cite[Remark 7]{rat}}] \label{n-th power of delta}
Let $\rho (f ^\ast)$ be the spectral radius of the linear self-map
$f^\ast \colon \NS (X)_\mathbb{R} \longrightarrow \NS (X)_\mathbb{R}.$
The dynamical degree $\delta_f$ is equal to the limit
$\lim_{n\to \infty} (\rho ((f ^n)^\ast ))^{1/n}.$
Thus we have $\delta_{f^n}=\delta_f^n$ for every $n\geq 1$.
\end{rmk}

Let $X_f(\var{k})$ be the set of $\var{k}$-rational points on $X$
at which $f^n$ is defined for every $n \geq 1.$
The {\it arithmetic degree} of $f$ at a $\overline{k}$-rational point $P\in X_f(\overline{k})$
is defined as follows.
Let $$h_H\colon X(\overline{k})\longrightarrow [0,\infty )$$
be an (absolute logarithmic) Weil height function associated with $H$
(see \cite[Theorem B3.2]{HS}).
We put
$$h_H^+(P):=\max\left\{ h_H(P),1\right\}.$$
We call
\begin{align*}
\overline{\alpha} _f(P)&:=\limsup_{n\to \infty} h_H^+(f^n(P))^{1/n},\text{ and}\\
\underline{\alpha} _f(P)&:=\liminf_{n\to \infty} h_H^+(f^n(P))^{1/n},\\
\end{align*}
{\it the upper arithmetic degree} and {\it the lower arithmetic degree}, respectively.
It is known that $\overline{\alpha}_f(P)$ and $\underline{\alpha}_f(P)$
do not depend on the choice of $H$ (see \cite[Proposition 12]{rat}).
If $\var{\alpha}_f(P)=\underline{\alpha}_f(P)$, the limit
$$\alpha _f(P):=\lim_{n\to \infty} h_H^+(f^n(P))^{1/n}$$
is called {\it the arithmetic degree of} $f$ {\it at} $P$.

\begin{rmk}\label{convergence}
When $f$ is an endomorphism, the existence of the limit
defining the arithmetic degree $\alpha_f(P)$ is proved
by Kawaguchi and Silverman in \cite[Theorem 3]{ab1}.
But it is not known in general.
\end{rmk}

\begin{rmk}\label{upperineq}
The inequality $\overline{\alpha}_f(P)\leq \delta_f$
is proved in \cite[Theorem 4]{rat} and \cite[Theorem 1.4]{Matsuzawa}.
Hence, in order to prove Conjecture \ref{KS}, it is enough
to prove the opposite inequality $\underline{\alpha}_f(P)\geq \delta_f$.
\end{rmk}

We recall the following result on relative dynamical degrees
proved by T.-C. Dinh and V.-A. Nguy\^en.

\begin{thm}[{\cite[Theorem 1.1]{dinh2}}]\label{dynamical}
Let $X$ and $Y$ be smooth projective varieties
defined over $\bb{C}$, with $\dim X \geq \dim Y$.
Let $f\colon X\dashrightarrow X,$ $g\colon Y\dashrightarrow Y$ and
$\pi \colon X \dashrightarrow Y$ be dominant rational maps
such that $\pi\circ f=g\circ \pi$.
Then we have
$$d_p(f)=\max_{\max\{0,p-\dim X+\dim Y \}\leq j\leq \min\{ p,\dim Y\}}d_j(g)d_{p-j}(f|_{\pi})$$
for every $0\leq p\leq \dim X.$
\end{thm}

Here, $d_p(f)$ and $d_{p}(f|_{\pi})$ are the {\it $p$-th dynamical degree}
and the {\it $p$-th relative dynamical degree}, respectively, defined in \cite[Section 3]{dinh2}.

\begin{cor}\label{Lemma:dynamical degree}
Let $k$ be a subfield of $\bb{C}$. Let $f\colon X \dashrightarrow X$
and $g\colon Y \dashrightarrow Y$
be dominant rational self-maps on smooth projective varieties
of dimension $\geq 1$ defined over $k$.
Let $f \times g \colon X \times Y \dashrightarrow X\times Y$ be the product of $f$ and $g$.
Then we have $\delta _{f \times g}=\max \left\{ \delta_f, \delta_g \right\}$.
\end{cor}
\begin{proof}
Since the (first) dynamical degrees do not change
when the base field $k$ is extended,
we may assume $k=\bb{C}$.
We apply Theorem \ref{dynamical} for $X\times Y$, $Y$, $f\times g$, $g$, $\pr_2$ and $p=1$.
We have $d_1(g)=\delta_g$, $d_0(g)=1$, $d_1((f\times g)|_{\pr_2})=\delta_f$,
$d_0((f\times g) |_{\pr_2})=1$, and $d_1(f\times g)=\delta_{f\times g}$ (see \cite[Section 3]{dinh2}).
Hence we get
\begin{align*}
\delta _{f\times g}
&= d_1(f\times g)\\
&= \max \{ d_1(g)d_0((f \times g)|_{\pr_2}), d_0(g)d_1((f \times g)|_{\pr_2})\} \\
&= \max \{ \delta_f, \delta_g \}.
\end{align*}
\end{proof}

\section{Some reductions of the Kawaguchi-Silverman conjecture.}\label{reduction}
In this section, we prove lemmata
which are useful to prove some cases of Conjecture \ref{KS}.

\begin{lem}\label{limit}
Let $\{a_n \}$ and $\{b_n\}$ be sequences of positive real numbers.
Assume that the limits $\lim_{n\to\infty}a_n^{1/n}$ and
$\lim_{n\to\infty}b_n^{1/n}$ exist and are not less than $1$.
Then the limit $\lim_{n\to\infty} (a_n+b_n)^{1/n}$ exists and
is equal to $\max \{\lim_{n\to\infty}a_n^{1/n}, \lim_{n\to\infty} b_n^{1/n}\}$.
\end{lem}
\begin{proof}
We put $\alpha:=\lim_{n\to\infty}a_n^{1/n}$ and $\beta:=\lim_{n\to\infty} b_n^{1/n}$.
If $\alpha > \beta$, we have
$$
a_n^{1/n} \leq (a_n+b_n)^{1/n} \leq (2a_n)^{1/n}
$$
for all sufficiently large $n$.
Hence the limit $\lim_{n\to\infty} (a_n+b_n)^{1/n}$ exists and is equal to $\alpha$.
The proof for the case $\alpha<\beta$ is similar.

If $\alpha=\beta$, since we have
$\lim_{n\to\infty} (a_n+b_n)^{1/n}=\alpha\cdot\lim_{n\to\infty}(1+b_n/a_n)^{1/n},$
it is enough to prove the assertion
when $a_n=1$ for all $n$ and $\beta=1$.
Fix a real number $0<\varepsilon<1.$
There exists an integer $n_0$
such that $b_n\leq (1+\varepsilon)^n$ holds for all $n\geq n_0$.
Then we have
$1\leq (1+b_n)^{1/n}\leq (1+(1+\varepsilon)^n)^{1/n}
\leq (2\cdot(1+\varepsilon)^n)^{1/n}=2^{1/n}(1+\varepsilon ).$
Hence we get
$1\leq \limsup_{n\to\infty} (1+b_n)^{1/n}\leq 1+\varepsilon$
and
$1\leq \liminf_{n\to\infty} (1+b_n)^{1/n}\leq 1+\varepsilon$.
Since the real number $0<\varepsilon <1$ is arbitrary, the limit
$\lim_{n\to\infty} (1+b_n)^{1/n}$ exists
and the assertion follows.
\end{proof}

\begin{lem}\label{prod} 
Let $X$ and $Y$ be smooth projective varieties defined over a number field $k$.
Let $f \colon X \longrightarrow X $ and $g\colon Y \longrightarrow Y$
be dominant endomorphisms defined over $k$, respectively.
Assume that Conjecture \ref{KS} is true for $f$ and $g$.
Then Conjecture \ref{KS} is true for the product endomorphism
$f\times g\colon X\times Y\longrightarrow X\times Y$.
\end{lem}
\begin{proof}
Let $D_1$ and $D_2$ be ample divisors on $X$ and $Y$, respectively.
Then $H:=\pr_1^\ast D_1+\pr_2^\ast D_2$ is an ample divisor
on $X\times Y$ (see \cite[I\hspace{-.1em}I. Proposition $7.10$]{AG}).
Fix the Weil height function associated with $H$,
$\pr_1^\ast D_1$ and $\pr_2^\ast D_2$ to satisfy that
$h_{H}=h_{\pr_1^\ast D_1}+h_{\pr_2^\ast D_2}.$
Since $D_1$ and $D_2$ are ample, we may assume that
$h_{\pr_1^\ast D_1}$ and $h_{\pr_2^\ast D_2}$ are positive functions.
When the forward $(f\times g)$-orbit $\mathcal{O}_{f\times g}(P)$
of $P\in (X\times Y)(\var{k})$ is Zariski dense in $X\times Y$,
the forward $f$-orbit and $g$-orbit of $\pr_1(P)$ and $\pr_2(P)$
are Zariski dense in $X$ and $Y$, respectively.
Since Conjecture \ref{KS} is true for $f$ and $g$,
for every $\var{k}$-rational point $P \in (X\times Y)(\var{k})$
whose forward $(f \times g)$-orbit $\mathcal{O}_{f\times g}(P)$
is Zariski dense in $X\times Y$, we have
\begin{align*}
\delta_{f\times g}&= \max\{\delta_f, \delta_g\}
		&&\text{by Corollary }\ref{Lemma:dynamical degree}\\
&= \max\{ \alpha_f(\pr_1 (P)), \alpha_g(\pr_2 (P))\}\\
&= \max\{ \lim_{n\to \infty} h^+_{\pr_1^\ast D_1}((f\times g)^n(P))^{1/n},
	 \lim_{n\to\infty} h^+_{\pr_2^\ast D_2}((f\times g)^n(P))^{1/n}\}\\
&= \lim_{n\to \infty} (h^+_{\pr_1^\ast D_1}((f\times g)^n(P))
	+ h^+_{\pr_2^\ast D_2}((f\times g)^n(P)))^{1/n}
		&&\text{by Lemma \ref{limit}}\\
&= \lim_{n\to\infty} h^+_{H}((f\times g)^n(P))^{1/n}\\
&=\alpha_{f\times g}(P).
\end{align*}
Hence Conjecture \ref{KS} is true for $f\times g$.
\end{proof}

\begin{lem}\label{iterate}
Let $X$ be a smooth projective variety defined over a number field $k$,
and $f\colon X \longrightarrow X$ an endomorphism defined over $k$.
Then Conjecture \ref{KS} is true for $f$
if and only if Conjecture \ref{KS} is true for $f^t$ for some $t\geq 1$.
\end{lem}
\begin{proof}
One direction is trivial.
Assume that Conjecture \ref{KS} is true for $f^t$ for some $t\geq 1$.
For every $\var{k}$-rational point $P\in X(\var{k})$, we have 
$$
\mathcal{O}_f(P) = \bigcup_{i=0}^{t-1}f^i(\mathcal{O}_{f^t}(P)).
$$
Therefore, $\mathcal{O}_f(P)$ is Zariski dense in $X$ if and only if
$\mathcal{O}_{f^t}(P)$ is Zariski dense in $X$.
Since we know the existence of $\alpha_f(P)$ (see Remark \ref{convergence}),
we get
\begin{align*}
\alpha_f(P)&= \lim_{n\to\infty} h_H^+(f^n(P))^{1/n}\\
&= \lim_{n\to \infty}h_H^+((f^{t})^n(P))^{1/tn}\\
&= \alpha_{f^t}(P)^{1/t}\\
&= \delta_{f^t}^{1/t}\\
&= \delta_f
		&& \text{by Remark \ref{n-th power of delta}}.
\end{align*}
Hence Conjecture \ref{KS} is true for $f$.
\end{proof}

\section{Splitting of endomorphisms on product varieties}\label{Splitting of endomorphisms}
In this section, we work over $\bb{C}$.
All varieties and morphisms are defined over $\bb{C}$.

Let $X$ and $Y$ be smooth projective varieties, and
$f\colon X\times Y\longrightarrow X\times Y$ a dominant endomorphism.
For a $\bb{C}$-rational point $x \in X(\bb{C})$,
let $i_x\colon Y\hookrightarrow X\times Y$ be
the closed embedding defined by $i_x(y):=(x,y)$.
For a $\bb{C}$-rational point $y\in Y(\bb{C})$,
let $j_y\colon X\hookrightarrow X\times Y$ be
the closed embedding defined by $j_y(x):=(x,y)$.

We say an endomorphism $f\colon X\times Y\longrightarrow X\times Y$ is {\it split}
if there exist endomorphisms $g\colon X\longrightarrow X$ and $h\colon Y\longrightarrow Y$
satisfying $f=g\times h$.
In this section, we study sufficient conditions
for endomorphisms on product varieties to be split.
\subsection{Some lemmata on endomorphisms on product varieties}\label{dominancy}

\begin{lem}\label{dominant}
Let $y\in Y(\bb{C})$ be a $\bb{C}$-rational point, and
$Z\subset X\times Y$ a closed subvariety with $\dim Z=\dim Y$.
Then the following are equivalent.
\begin{itemize}
\item $\pr_2|_Z\colon Z\longrightarrow Y$ is dominant.
\item $\deg(j_y(X)\cdot Z)\neq 0$.
			Here, $j_y(X)\cdot Z$ is a zero cycle on $X\times Y$,
			which is the intersection product of cycles $j_y(X)$ and $Z$ on $X\times Y$.
\end{itemize}
\end{lem}

\begin{proof}
If $\pr_2|_Z$ is dominant, we have $\pr_{2,\ast}(Z)=\deg(Z/Y)Y$
as cycles on $Y$ (see \cite[1.4]{Fulton}).
Since $j_y(X)=\pr_2^\ast (y)$ as cycles on $X\times Y$, we have
\begin{align*}
\deg(j_y(X) \cdot Z) &= \deg (\pr_{2}^\ast(y)\cdot Z)\\
&= \deg(y\cdot \pr_{2,\ast}(Z))
		&& \text{by the projection formula (\cite[Example 8.1.7]{Fulton})}\\
&= \deg(Z/Y)\\
&\geq 1.
\end{align*}
Conversely, if $\pr_2|_Z$ is not dominant, $\pr_2|_Z$ is not surjective.
We may assume $y\not \in \pr_2(Z)$
because $\deg(j_y(X)\cdot Z)$ does not depend on $y$;
see Remark \ref{indep} below.
Since $j_y(X)\cap Z=\emptyset$, we have $\deg(j_y(X)\cdot Z)=0$.
\end{proof}

\begin{rmk}\label{indep}
For $y,y'\in Y(\bb{C})$,
the cycles $j_y(X)$ and $j_{y'}(Y)$ on $X\times Y$ are algebraically equivalent.
Hence the degree $\deg(j_y(X)\cdot Z)$
does not depend on $y$ (see \cite[Corollary 10.2.2]{Fulton}).
\end{rmk}

\begin{lem}\label{dominant2}
There is a positive integer $t\geq 1$ such that the endomorphism
$$\pr_2\circ f^t\circ i_x\colon Y\longrightarrow Y$$
is dominant for every $x\in X(\bb{C})$.
\end{lem}
\begin{proof}
Fix $\bb{C}$-rational points $x_0\in X(\bb{C})$ and $y_0 \in Y(\bb{C})$.
We put $Z_{x_0,t}:=f^t(i_{x_0}(Y))$ for each $t\geq 0$.
We have $\dim Z_{x_0,t}=\dim Y$ because $f$ is finite
(see \cite[Lemma 1]{Beau2}, \cite[Lemma 2.3 (1)]{Fujim}).
Fix $t\geq 1$.
Let us consider the cohomology classes
$\cl(j_{y_0}(X))\in H^{2\dim Y}(X(\bb{C})\times Y(\bb{C}),\bb{Q})$ and
$\cl (Z_{x_0,t})\in H^{2\dim X}(X(\bb{C})\times Y(\bb{C}),\bb{Q})$.
Recall that the intersection number is calculated
by the cup product of cohomology classes.
Hence, we have $\cl(j_{y_0}(X))\cup \cl(Z_{x_0,t})=\deg(j_{y_0}(X)\cdot Z_{x_0,t})$
in $H^{2\dim X+2\dim Y}(X(\bb{C})\times Y(\bb{C}), \bb{Q})=\bb{Q}.$
By Lemma \ref{dominant}, $\pr_2\circ f^t\circ i_{x_0}$ is dominant
if and only if $\deg (j_{y_0}(X)\cdot Z_{x_0,t})\neq 0.$
Since $\cl(Z_{x_0,t})$ does not depend on $x_0$,
we see that $\pr_2\circ f^t \circ i_x$ is dominant for one $x\in X(\bb{C})$
if and only if it is dominant for every $x\in X(\bb{C})$.
Therefore, it is enough to consider the case of $\pr_2\circ f^t \circ i_{x_0}.$

Since $H^{2\dim X}(X(\bb{C})\times Y(\bb{C}), \bb{Q})$ is
a finite dimensional $\bb{Q}$-vector space,
there is an unique integer $s\geq 0$ such that
$\cl(Z_{x_0,i})$ $(0\leq i \leq s)$ are linearly independent over $\bb{Q}$,
but $\cl(Z_{x_0,i})$ $(0\leq i\leq s+1)$ are linearly dependent over $\bb{Q}$,
where note that $\cl(Z_{x_0,i})\neq 0$
since $X\times Y$ is projective and $Z_{x_0,i}$ is irreducible.
We write
$\cl(Z_{x_0,s+1})=\sum_{i=0}^s a_i\cl(Z_{x_0,i})$
for some $a_i\in \bb{Q}.$

Assume that $\pr_2\circ f^t \circ i_{x_0}$ is not dominant for all $t\geq 1$.
Then we have $\cl(j_{y_0}(X))\cup \cl(Z_{x_0,t})=0$ for all $t\geq 1.$
On the other hand, we again have $\cl(j_{y_0}(X))\cup \cl(Z_{x_0,0})=1.$
Calculating the cup products with $\cl(j_{y_0}(X)),$ we have
$$0=\cl(j_{y_0}(X))\cup\cl(Z_{x_0,s+1})=\sum_{i=0}^s a_i(\cl(j_{y_0}(X))\cup \cl(Z_{x_0,i}))=a_0.$$

If $s=0$, we get $\cl(Z_{x_0,1})=0$.
But this is a contradiction because $\cl(Z_{x_0,1})$ is a prime cycle.
Thus we may assume $s\geq 1$.

Recall that $f$ induces a bijective $\bb{Q}$-linear map
$$f_\ast\colon H^{2\dim X}(X(\bb{C})\times Y(\bb{C}),\bb{Q})
\overset{\cong}{\longrightarrow} H^{2\dim X}(X(\bb{C})\times Y(\bb{C}),\bb{Q})$$
satisfying
$f_\ast(\cl(Z_{x_0,t}))=\cl(f_\ast (Z_{x_0,t}))=\deg(Z_{x_0,t}/Z_{x_0,t+1})\cdot \cl(Z_{x_0,t+1})$
(see \cite[Lemma 1]{Beau2}). We put
$$\alpha:=\sum_{i=1}^s a_i\cdot \deg(Z_{x_0,i-1}/Z_{x_0,i})^{-1}\cdot \cl(Z_{x_0,i-1}).$$
Since $a_0=0$, we have
\begin{align*}
f_\ast(\alpha)
&= \sum_{i=1}^s a_i\cl(Z_{x_0,i})\\
&= \cl(Z_{x_0,s+1})\\
&= f_*(\deg(Z_{x_0,s}/Z_{x_0,s+1})^{-1}\cdot \cl(Z_{x_0,s})).
\end{align*}
By the injectivity of $f_\ast$, we have
$\alpha=\deg(Z_{x_0,s}/Z_{x_0,s+1})^{-1}\cdot \cl(Z_{x_0,s}).$
But it contradicts with the assumption that
$\cl(Z_{x_0,i})$ $(0\leq i\leq s)$ are linearly independent over $\bb{Q}$.
Therefore, $\pr_2\circ f^t\circ i_{x_0}$ is dominant for some $t\geq 1$.
\end{proof}

\subsection{Splitting of endomorphisms on product varieties (1)}\label{important1}

For a smooth projective variety $X$, the automorphism group of $X$,
denoted by $\Aut(X)$,
has a natural structure of a group scheme locally of finite type (see \cite[Theorem 3.7]{MO}).
Its neutral component is denoted by $\Aut^\circ(X).$
Let $\Sur(X)$ be the scheme of surjective endomorphisms on $X$,
which has a natural action of $\Aut(X)$ (for details, see \cite{Fujim}, \cite{Brion}).

\begin{lem}
\label{Betti}
Let $X$ be a smooth projective variety.
\begin{enumerate}
\item $b_1(X)=0$
if and only if $\Alb(X)=0$.
\item If $b_1(X) = 0$, $\Aut^\circ(X)$ is a linear algebraic group.
\end{enumerate}
\end{lem}

\begin{proof}
The part (1) follows from the equality
$b_1(X) = 2 \cdot \dim \Alb(X)$
(see \cite[Theorem V.13]{Beau1}).
For the part (2), see \cite[Corollary 5.8]{Fujiki}, \cite[p.73, Remark 2]{Brion}.
\end{proof}

\begin{lem}
\label{Fibration}
Let $X$ and $Y$ be smooth projective varieties,
and $f \colon X \times Y \longrightarrow X \times Y$ a dominant endomorphism.
Assume that at least one of the following conditions is satisfied:
\begin{itemize}
\item $b_1(X) = 0$, or
\item $\Aut^\circ(Y)$ is a linear algebraic group.
\end{itemize}
Then there is an integer $t \geq 1$ such that
$f^t(x,y) = (g(x,y),h(y))$
for some morphisms
$g \colon X \times Y \longrightarrow X$ and $h \colon Y \longrightarrow Y$.
\end{lem}

\begin{proof}
Take an integer $t \geq 1$ as in Lemma \ref{dominant2}.
Fix a $\bb{C}$-rational point $x_0 \in X(\bb{C})$,
and get a holomorphic map
$X(\bb{C}) \longrightarrow \Sur(Y)(\bb{C}), x \mapsto \pr_2 \circ f^{t} \circ i_x.$
The image of $X(\bb{C})$ in $\Sur(Y)(\bb{C})$ is contained
in a left $\Aut^\circ(Y)(\bb{C})$-orbit by Horst's theorem \cite[Theorem 3.1]{Horst}.
Hence there is a holomorphic map
$\varphi \colon X(\bb{C}) \longrightarrow \Aut^\circ(Y)(\bb{C})$ satisfying
$\pr_2 \circ f^t \circ i_x = \varphi(x) \circ \pr_2 \circ f^{t} \circ i_{x_0}$ for all $x \in X(\bb{C})$.

We shall show $\varphi$ is constant.
Since $\Aut^\circ(Y)$ is a connected algebraic group,
by Chevalley's theorem (\cite{Chevalley}),
there is a linear normal subgroup $\Aut^\circ(Y)_{\mathrm{lin}} \subset \Aut^\circ(Y)$
such that $\Aut^\circ(Y)/\Aut^\circ(Y)_{\mathrm{lin}}$
is an abelian variety.
If $b_1(X) = 0$, $\Alb(X)$ is trivial by Lemma \ref{Betti} (1).
By the universal property of the Albanese varieties (see \cite[Theorem V.13]{Beau1}),
the image of $X(\bb{C})$ in $(\Aut^\circ(Y)/\Aut^\circ(Y)_{\mathrm{lin}})(\bb{C})$ is trivial.
Hence $\varphi$ is a holomorphic map
from $X(\bb{C})$ to $(\Aut^\circ(Y)_{\mathrm{lin}})(\bb{C})$.
Since $X$ is a projective variety
and $\Aut^\circ(Y)_{\mathrm{lin}}$ is an affine variety,
$\varphi$ is constant.
Similarly, if $\Aut^\circ(Y)$ is a linear algebraic group,
we have $\Aut^\circ(Y) = \Aut^\circ(Y)_{\mathrm{lin}}$,
and $\varphi$ is constant.

Therefore, we conclude $\varphi(x) = \id_Y$ for all $x \in X(\bb{C})$.
Putting $g := \pr_1 \circ f^{t}$ and $h := \pr_2 \circ f^{t} \circ i_{x_0}$,
we have
$f^t(x,y) = (g(x,y),h(y))$.
\end{proof}

\begin{thm}\label{Splitting1}
Let $X$ and $Y$ be smooth projective varieties,
and $f\colon X\times Y \longrightarrow X\times Y$
a dominant endomorphism.
Assume that at least one of the following conditions is satisfied:
\begin{itemize}
\item $b_1(X)=0$,
\item $b_1(Y)=0$, or
\item $\Aut^{\circ}(X)$ and $\Aut^{\circ}(Y)$ are linear algebraic groups.
\end{itemize}
Then $f^t$ is split for some $t \geq 1$.
\end{thm}

\begin{proof}
By changing the role of $X$ and $Y$, we may assume that
the first condition or the third condition is satisfied.
Take an integer $t\geq 1$ as in Lemma \ref{Fibration}.
Fix $\bb{C}$-rational points $x_0\in X(\bb{C})$ and $y_0\in Y(\bb{C})$.
We have $f^t(x,y)=(g_1(x,y),h_1(y))$ for some morphisms
$g_1\colon X\times Y\longrightarrow X$
and $h_1\colon Y\longrightarrow Y$.

For each $y \in Y(\bb{C}),$
the endomorphism $g_1\circ j_y\colon X\longrightarrow X$ is finite
because it is the restriction of $f^t$ to $j_y(X)\subset X\times Y$
and $f^t$ is finite (see \cite[Lemma 1]{Beau2}, \cite[Lemma 2.3 (1)]{Fujim}).
Hence $g_1\circ j_y$ is surjective for each $y\in Y(\bb{C})$.
We have a holomorphic map
$Y(\bb{C}) \longrightarrow \Sur(X)(\bb{C}), y \mapsto g_1\circ j_y.$
The image of $Y(\bb{C})$ in $\Sur(X)(\bb{C})$ is contained
in a left $\Aut^\circ(X)(\bb{C})$-orbit by Horst's theorem \cite[Theorem 3.1]{Horst}.
We have a holomorphic map
$\psi \colon Y(\bb{C}) \longrightarrow \Aut^\circ(X)(\bb{C})$
satisfying
$g_1\circ j_y=\psi (y)\circ g_1 \circ j_{y_0}$
for all $y \in Y(\bb{C})$.
If $b_1(X) = 0$, $\Aut^\circ(X)$ is a linear algebraic group
by Lemma \ref{Betti} (2).
Since $Y$ is a projective variety and $\Aut^\circ(X)$ is an affine variety,
$\psi$ is constant.
Similarly, if $\Aut^\circ(X)$ is a linear algebraic group, $\psi$ is constant.

Therefore, we conclude $\psi(y) = \id_X$ for all $y \in Y(\bb{C})$.
Putting $g(x) = g_1(x,y_0)$ and $h(y) = h_1(y)$,
we have $f^t = g\times h$.
\end{proof}

\subsection{Splitting of endomorphisms on product varieties (2)}\label{important2}
In this subsection, we consider smooth projective varieties $X$ and $Y$ satisfying
\begin{equation*}
\rank \NS(X \times Y) = \rank \NS(X) + \rank \NS(Y).
\end{equation*}
Under this assumption, we shall prove results similar to Theorem \ref{Splitting1}
for dominant endomorphisms on $X \times Y$.

\begin{lem}
\label{NS}
Let $X$ and $Y$ be smooth projective varieties satisfying
\begin{equation*}
\rank \NS(X \times Y) = \rank \NS(X) + \rank \NS(Y).
\end{equation*}
Fix $\bb{C}$-rational points $x_0 \in X(\bb{C})$ and $y_0 \in Y(\bb{C})$.
\begin{enumerate}
\item The following maps are isomorphisms and inverses to each other:
\begin{align*}
\NS(X)_{\bb{R}} \oplus \NS(Y)_{\bb{R}} \overset{\cong}{\longrightarrow} \NS(X \times Y)_{\bb{R}},
  &\qquad (\alpha,\beta) \mapsto \pr_1^{\ast} \alpha + \pr_2^{\ast} \beta, \\
\NS(X \times Y)_{\bb{R}} \overset{\cong}{\longrightarrow} \NS(X)_{\bb{R}} \oplus \NS(Y)_{\bb{R}},
  &\qquad \gamma \mapsto (j_{y_0}^{\ast} \gamma, i_{x_0}^{\ast} \gamma).
\end{align*}
\item For a closed subvariety $\iota_Z \colon Z \hookrightarrow X$
		and an element $\alpha \in \NS(X \times Y)_{\bb{R}}$,
the following are equivalent:
\begin{itemize}
\item $(\iota_Z \times \id_Y)^{\ast} \alpha \in \NS(Z \times Y)_{\bb{R}}$ is ample.
\item $(j_{y_0} \circ \iota_Z)^{\ast} \alpha \in \NS(Z)_{\bb{R}}$
		and $i_{x_0}^{\ast} \alpha \in \NS(Y)_{\bb{R}}$ are ample.
\end{itemize}
\end{enumerate}
\end{lem}

\begin{proof}
Let $p$ and $\sigma$ be the first map and the second map, respectively,
in our statement of part (1).
Since $\sigma\circ p$ is identity map, $p$ is injective and $\sigma$ is surjective.
Furthermore, since the vector spaces $\NS(X)_{\bb{R}} \oplus \NS(Y)_{\bb{R}}$
and $\NS(X \times Y)_{\bb{R}}$ have the same dimension,
the maps $p$ and $\sigma$ are isomorphisms.
This is the assertion of part (1).
For part (2), by (1), we have
$\alpha = \pr_1^{\ast} j_{y_0}^{\ast} \alpha + \pr_2^{\ast} i_{x_0}^{\ast} \alpha$.
Hence we have
\begin{align*}
 (\iota_Z \times \id_Y)^{\ast} \alpha
 	&= (\pr_1 \circ (\iota_Z \times \id_Y))^{\ast} (j_{y_0}^{\ast} \alpha)
 			+(\pr_2 \circ (\iota_Z \times \id_Y))^{\ast} (i_{x_0}^{\ast} \alpha)\\
	&= \pr_Z^{\ast} ((j_{y_0} \circ \iota_Z)^{\ast} \alpha)
 		+ \pr_Y^{\ast} (i_{x_0}^{\ast} \alpha),
\end{align*}
where $\pr_Z \colon Z\times Y\longrightarrow Z$ and
$\pr_Y \colon Z\times Y\longrightarrow Y$ are projections.
Therefore,
$(\iota_Z \times \id_Y)^{\ast} \alpha$ is ample if and only if
$(j_{y_0} \circ \iota_Z)^{\ast} \alpha$ and $i_{x_0}^{\ast} \alpha$ are ample
(see \cite[Proposition 7.10]{AG}). 
\end{proof}

\begin{lem}
\label{DomBlock}
Let $X$ and $Y$ be smooth projective varieties satisfying
\begin{equation*}
\rank \NS(X \times Y) = \rank \NS(X) + \rank \NS(Y),
\end{equation*}
and $f \colon X \times Y \longrightarrow X \times Y$
an endomorphism.
If $\pr_2 \circ f \circ i_{x_0} \colon Y \longrightarrow Y$
is dominant for some $x_0 \in X(\bb{C})$,
we have
$f(x,y) = (g(x,y),h(y))$
for some morphisms
$g \colon X \times Y \longrightarrow X$ and $h \colon Y \longrightarrow Y$.
\end{lem}

\begin{proof}
It suffices to prove 
$\pr_2 \circ f \circ j_{y_0} \colon X \longrightarrow Y$
is constant for each $y_0$.
Assume that it is not constant.
Since $\dim \pr_2 \circ f \circ j_{y_0}(X) \geq 1$,
there is an irreducible curve $\iota_Z \colon Z \hookrightarrow X$ with
$\dim \pr_2 \circ f \circ j_{y_0}(Z) = 1$.
Let $\delta \in \NS(Y)_{\bb{R}}$ be the class of an ample divisor on $Y$.
Since
$\pr_2 \circ f \circ j_{y_0} \circ \iota_Z \colon Z \longrightarrow Y$
is finite, the pullback
$$
(\pr_2 \circ f \circ j_{y_0} \circ \iota_Z)^{\ast} \delta
= (j_{y_0} \circ \iota_Z)^{\ast} (\pr_2 \circ f)^{\ast} \delta
\in \NS(Z)_{\bb{R}}
$$
is ample.
On the other hand, since $\pr_2 \circ f \circ i_{x_0}$ is a dominant endomorphism on $Y$,
it is finite (see \cite[Lemma 1]{Beau2}, \cite[Lemma 2.3 (1)]{Fujim}),
and the pullback
$$
(\pr_2 \circ f \circ i_{x_0})^{\ast} \delta
= (i_{x_0})^{\ast} (\pr_2 \circ f)^{\ast} \delta
\in \NS(Y)_{\bb{R}}
$$
is ample.
Applying Lemma \ref{NS} (2) for
$\alpha = (\pr_2 \circ f)^{\ast} \delta$,
we see that
$$
(\iota_Z \times \id_Y)^{\ast} (\pr_2 \circ f)^{\ast} \delta
= (\pr_2 \circ f \circ (\iota_Z \times \id_Y))^{\ast} \delta
\in \NS(Z \times Y)_{\bb{R}}
$$
is ample.
Hence the morphism
$\pr_2 \circ f \circ (\iota_Z \times \id_Y) \colon Z \times Y \longrightarrow Y$
must be finite by the projection formula.
It is a contradiction because $\dim(Z\times Y)>\dim Y$.
Hence $\pr_2 \circ f \circ j_{y_0}$ is constant as required.
\end{proof}

\begin{thm}
\label{Splitting2}
Let $X$ and $Y$ be smooth projective varieties satisfying
\begin{equation*}
\rank \NS(X \times Y) = \rank \NS(X) + \rank \NS(Y)
\end{equation*}
and $f \colon X \times Y \longrightarrow X \times Y$
an endomorphism.
Then $f^t$ is split for some $t\geq 1$.
\end{thm}

\begin{proof}
Applying Lemma \ref{dominant2} and Lemma \ref{DomBlock},
there is an integer $t_1 \geq 1$ such that
$f^{t_1}(x,y) = (g_1(x,y),h_1(y))$ for some morphisms
$g_1\colon X \times Y \longrightarrow X$ and
$h_1\colon Y \longrightarrow Y$.
Changing the role of $X$ and $Y$ and applying Lemma \ref{dominant2}
and Lemma \ref{DomBlock} again,
there is an integer $t_2 \geq 1$ such that
$f^{t_1 t_2}(x,y) = (g_2(x),h_2(x,y))$ for some morphisms
$g_2\colon X \longrightarrow X$ and
$h_2\colon X \times Y \longrightarrow Y$.
Then we have $f^{t_1t_2}(x,y)=(g_2(x),h_1^{t_2}(y)).$
Hence $f^{t_1t_2}$ is split.
\end{proof}

\section{Proof of Theorem {\ref{main1}}}\label{cor1}
The base field $k$ is a number field in this section.

For a smooth projective variety $X$ defined over $k$,
we say ``\textit{Conjecture \ref{KS} is true for endomorphisms on $X$}''
if Conjecture \ref{KS} is true for
every dominant endomorphism $f \colon X \longrightarrow X$
defined over a finite extension of $k$
and every $\var{k}$-rational point $P \in X(\var{k})$
with Zariski dense forward $f$-orbit.

\begin{lem}\label{reduction of splitting}
Let $X$ and $Y$ be smooth projective varieties defined over $k$,
and $f \colon X \times Y\longrightarrow X\times Y$
an endomorphism defined over $k$.
Fix an embedding $k \hookrightarrow \bb{C}$.
Assume that $f$ is split over $\bb{C}$, i.e.,
there exist endomorphisms $g\colon X\longrightarrow X$ and $h\colon Y\longrightarrow Y$
defined over $\bb{C}$ satisfying $f=g\times h$.
Then, there is a finite extension $k'/k$ such that $f$ is split over $k'$.
\end{lem}

\begin{proof}
Take $\overline{k}$-rational points $x_0 \in X(\overline{k})$ and $y_0 \in Y(\overline{k})$.
We put $g_0 := \pr_1 \circ f \circ j_{y_0}$ and $h_0 = \pr_2 \circ f \circ i_{x_0}$.
These are endomorphisms on $X$ and $Y$, respectively, defined over $\overline{k}$.
By assumption, $f(x,y) = (g_0(x), h_0(y))$ for all $x \in X(\bb{C})$ and $y \in Y(\bb{C})$.
Hence $f = g_0 \times h_0$ as endomorphisms defined over $\overline{k}$.
Since $g_0$ and $h_0$ are defined over a finite extension of $k$,
the assertion follows.
\end{proof}

\begin{lem}\label{Lemma:strong}
Let $X$ and $Y$ be smooth projective varieties defined over $k$.
Assume that Conjecture \ref{KS} is true for endomorphisms on $X$ and $Y$.
Moreover, assume that at least one of the following conditions is satisfied:
\begin{itemize}
\item $b_1(X) = 0$,
\item $b_1(Y) = 0$,
\item $\Aut^0(X)$ and $\Aut^0(Y)$ are linear algebraic groups, or
\item $\rank \NS(X \times Y) = \rank \NS(X) + \rank \NS(Y)$.
\end{itemize}
Then for any endomorphisms on $X\times Y$, there is an integer $t\geq 1$ such that
the iteration $f^t$ splits.
Moreover, Conjecture \ref{KS} is true for endomorphisms on $X \times Y$.
Moreover if $b_1(X)=b_1(Y)=0,$ we have $b_1(X\times Y)=0.$
\end{lem}

\begin{proof}
Let $f \colon X \times Y\longrightarrow X\times Y$
be an endomorphism defined over $k$.
If one of the first, second, or third condition is satisfied,
we apply Theorem \ref{Splitting1} to conclude that
$f^t$ is split for some $t\geq 1$.
If the fourth condition is satisfied, we apply Theorem \ref{Splitting2}
to conclude that $f^t$ is split over $\bb{C}$ for some $t \geq 1$.
In any of these cases, by Lemma \ref{reduction of splitting}, we have $f^t=g\times h$
for endomorphisms $g\colon X\longrightarrow X$ and
$h \colon Y\longrightarrow Y$ defined over a finite extension $k'$ of $k$.

Since Conjecture \ref{KS} is true for $g$ and $h$, it is true for $f^t$ by Lemma \ref{prod}.
Therefore, it is true for $f$ by Lemma \ref{iterate}.

Finally, the assertion on the first Betti number follows from the equality
$b_1(X\times Y)=b_1(X)+b_1(Y).$
\end{proof}

\begin{lem}\label{surface}
Let $X$ be an Enriques or a K3 surface defined over $k$.
Then $b_1(X)=0$ and
Conjecture \ref{KS} is true for endomorphisms on $X$.
\end{lem}
\begin{proof}
It is well-known that the first Betti number of an Enriques or a K3 surface is zero
(see \cite[Theorem VIII. 2]{Beau1}).
It is well-known that any dominant endomorphism on an Enriques or a K3 surface
is an automorphism.
This fact is proved by K. Peters in \cite[Satz 9]{Peters}
(see also \cite[Corollary 2.4]{Fujim} or \cite[Theorem 7.6.11]{Kob}).
Since Conjecture \ref{KS} is true for automorphisms
on smooth projective surfaces (\cite[Theorem 2 (c)]{eg}),
Conjecture \ref{KS} is true for endomorphisms on an Enriques or a K3 surface.
\end{proof}

Now, we shall complete the proof of Theorem \ref{main1}.

\begin{proof}[Proof of Theorem \ref{main1}]
Recall that each $X_i$ satisfies at least one of the following conditions:
\begin{itemize}
\item $b_1(X_i) = 0$ and $\rank \NS(X_i) = 1$,
\item $X_i$ is an abelian variety,
\item $X_i$ is an Enriques surface, or
\item $X_i$ is a $K3$ surface.
\end{itemize}

Let $I$ be the set of indices where $X_i$ is an abelian variety.
We put $Y := \prod_{i \in I} X_i$ and $Z = \prod_{i \notin I} X_i$.

Then $Y$ is an abelian variety of dimension $\sum_{i \in I} \dim X_i$.
Conjecture \ref{KS} is true for endomorphisms on $Y$
by \cite[Corollary 32]{ab1}, \cite[Theorem 2]{ab2}.

Note that $b_1(X_i)=0$ for $i\not\in I$ by Lemma \ref{surface}.
Then Conjecture \ref{KS} is true for any endomorphisms on each $X_i$
by \cite[Theorem 2 (a)]{eg} and Lemma \ref{surface},
so Lemma \ref{Lemma:strong} shows
Conjecture \ref{KS} for any endomorphisms on $Z$ and $b_1(Z)=0$.

By Lemma \ref{Lemma:strong}, Conjecture \ref{KS} is true for endomorphisms on $Y \times Z$.
\end{proof}

\section{Proof of Theorem {\ref{main2}}}\label{cor2}
\begin{proof}[Proof of Theorem \ref{main2}]
We may assume $k=\bb{C}$ and $Y$ is of general type.
Then the automorphism group $\Aut(Y)$ is finite
(\cite[Corollary 2]{Matsumura}, \cite[Theorem 11.12]{Iitaka}).
Hence the neutral component $\Aut^\circ(Y)$ is trivial.

Applying Lemma \ref{Fibration}, there is an integer $t \geq 1$
such that $f^t(x,y) = (g(x,y),h(y))$ for some morphisms
$g\colon X\times Y \longrightarrow X$ and
$h\colon Y\longrightarrow Y$.

If $h$ is not surjective, the assertion is trivial. Thus, we may assume $h$ is surjective.
Since $Y$ is of general type,
$h$ is an automorphism (see \cite[Theorem 1]{KO}, \cite[Theorem 7.6.1]{Kob}, or \cite[Proposition 2.6]{Fujim}).
Replacing $t$ by $t \cdot |\Aut(Y)|$,
we may assume $h = \id_Y$.

Then, for every $\bb{C}$-rational point $P \in (X \times Y)(\bb{C})$,
we have $\pr_2 (P) = \pr_2 (f^{t}(P))$.
The forward $f$-orbit $\mathcal{O}_f(P)$ is contained in the set
$$ \{\, Q \in (X \times Y)(\bb{C}) : \pr_2 (Q) = \pr_2 (f^{j}(P)) \ \text{for some}\ 0 \leq j < t \,\},$$
which is not Zariski dense in $X \times Y$
(in fact, this is of dimension $\leq \dim X$).
\end{proof}

\section*{acknowledgements}
This paper is a revised version of the master's thesis of the author.
In writing this paper,
Tetsushi Ito made a polite and delicate guidance for the author.
The author would like to thank Shu Kawaguchi
and Serge Cantat
for their advice on the definition and properties of dynamical degrees.
It is very kind of them to assist the author to complete this paper.

\end{document}